\documentclass[11pt]{amsart}
\usepackage{amssymb,hyperref,latexsym}

\newtheorem{theorem}{Theorem}[section]
\newtheorem{corollary}[theorem]{Corollary}

\newtheorem{lemma}[theorem]{Lemma}
\newtheorem{conjecture}[theorem]{Conjecture}

\theoremstyle{definition}
\newtheorem{definition}[theorem]{Definition}
\newtheorem{notation}[theorem]{Notation}

\DeclareMathSymbol\unlhd{\mathrel}{lasy}{"02}

\def\Z#1{{\bf Z}(#1)}
\def\cent#1#2{{\bf C}_{#1}(#2)}
\def\norm#1#2{{\bf N}_{#1}(#2)}

\def\Aut{{\rm Aut}}
\def\Out{{\rm Out}}
\def\Cay{{\rm Cay}}
\def\Dic{{\rm Dic}}
\newcommand{\D}{\mathrm{D}}
\newcommand{\C}{\mathrm{C}}
\newcommand{\Q}{\mathrm{Q}}

\newcommand{\Sym}{\mathop{\mathrm{Sym}}}
\newcommand{\PSL}{\mathop{\mathrm{PSL}}}
\newcommand{\PGammaL}{\mathop{\mathrm{P}\Gamma\mathrm{L}}}
\newcommand{\PGL}{\mathop{\mathrm{PGL}}}

\renewcommand{\wr}{\mathop{\mathrm{wr}}}

\begin{document}

\title[Cayley graphs on generalised dicyclic groups]{Automorphisms of Cayley graphs on generalised dicyclic groups} 

\author[J. Morris]{Joy Morris}
\address{Joy Morris, Department of Mathematics and Computer Science,
University of Lethbridge, Lethbridge, AB. T1K 3M4. Canada}
\email{joy@cs.uleth.ca}

\author[P. Spiga]{Pablo Spiga}
\address{Pablo Spiga,
Dipartimento di Matematica e Applicazioni, University of Milano-Bicocca, Via Cozzi 53, 20125 Milano, Italy}\email{pablo.spiga@unimib.it}

\author[G. Verret]{Gabriel Verret}
\address{Gabriel Verret, Centre for Mathematics of Symmetry and Computation, School of Mathematics and Statistics, The University of Western Australia, 35 Stirling Highway, Crawley, WA 6009, Australia. \newline
\indent Also affiliated with : UP FAMNIT, University of Primorska, Glagolja\v{s}ka 8, 6000 Koper, Slovenia.}\email{gabriel.verret@uwa.edu.au}

\thanks{The first author is supported in part by the National Science and Engineering Research Council of Canada. The third author is supported by UWA as part of the Australian Research Council grant DE130101001.\\
 \indent Address correspondence to Pablo Spiga. (pablo.spiga@unimib.it)}


\begin{abstract}
A graph is called a \emph{GRR} if its automorphism group acts regularly on its vertex-set. Such a graph is necessarily a Cayley graph. Godsil has shown that there are only two infinite families of finite groups that do not admit GRRs : abelian groups and generalised dicyclic groups~\cite{Godsil}. Indeed, any Cayley graph on such a group admits specific additional graph automorphisms that depend only on the group. Recently, Dobson and the last two authors showed that almost all Cayley graphs on abelian groups admit no automorphisms other than these obvious necessary ones~\cite{DSV}. In this paper, we prove the analogous result for Cayley graphs  on the remaining family of exceptional groups: generalised dicyclic groups.
\end{abstract}
\maketitle

\section{Introduction}

In this paper, all groups considered are finite and all graphs are finite, undirected, and have no multiple edges. (They may have loops and they may be disconnected.) Let $R$ be a group and let $S$ be an inverse-closed subset of $G$. The \emph{Cayley graph} on $R$ with connection set $S$, denoted $\Cay(R,S)$, is the graph with vertex-set $R$ and with $\{g,h\}$ being an edge  if and only if $gh^{-1}\in S$. It is easy to check that $R$ acts regularly, by right multiplication, as a group of automorphisms of $\Cay(R,S)$. If, in fact, $R$ is the full automorphism group of $\Cay(R,S)$ then $\Cay(R,S)$ is called a \emph{GRR} (for graphical regular representation).

The most natural question concerning GRRs is to determine which groups admit GRRs. This question was answered by Godsil~\cite{Godsil}, after a long series of partial results by various authors (see~\cite{Hetzel,Imrich,NowWat} for example). 

It turns out that there are only two infinite families of groups which do not admit GRRs. The first family consists of abelian groups of exponent greater than two.  If $A$ is such a group and $\iota$ is the automorphism of $A$ mapping every element to its inverse then  every Cayley graph on $A$ admits $A\rtimes\langle\iota\rangle$ as a group of automorphisms. Since $A$ has exponent greater than $2$, $\iota\neq 1$ and hence no Cayley graph on $A$ is a GRR. 

The other groups that do not admit GRRs are the generalised dicyclic groups, which we now define.

\begin{definition}
Let $A$ be an abelian group of even order and of exponent greater than $2$, and let $y$ be an involution of $A$. The \emph{generalised dicyclic group} $\Dic(A,y,x)$ is the group $\langle A,x\mid x^2=y, a^x=a^{-1}, \forall a \in A\rangle$. A group is called \emph{generalised dicyclic} if it is isomorphic to some $\Dic(A,y,x)$. When $A$ is cyclic, $\Dic(A,y,x)$ is called a \emph{dicyclic} or \emph{generalised quaternion group}.
\end{definition}

The importance of generalised dicyclic groups in this context stems from the fact that, just like abelian groups, they admit a non-trivial group automorphism $\iota$ that maps every group element either to itself or to its inverse. More details on this and other basic facts concerning generalised dicyclic groups can be found in Subsection~\ref{Sec:GenDicyclic}. For the moment, it suffices to observe that the existence of this group automorphism $\iota$ implies that no Cayley graph on a generalised dicyclic group is a GRR.

As mentioned earlier, it was proved by Godsil that abelian and generalised dicyclic groups are the only two infinite families of groups which do not admit GRRs. The stronger conjecture that follows was made by Babai, Godsil, Imrich and Lov\'asz~\cite[Conjecture 2.1]{BaGo}.

\begin{conjecture}\label{graphmainconjecture}
Let $R$ be a group of order $n$ which is neither generalised dicyclic nor abelian. The proportion of inverse-closed subsets $S$ of $R$ such that $\Cay(R,S)$ is a GRR goes to $1$ as $n\to\infty$.
\end{conjecture}

Conjecture \ref{graphmainconjecture} has been verified in the case that $R$ is nilpotent of odd order by  Babai and Godsil~\cite[Theorem 2.2]{BaGo}.

If $A$ is abelian then, as we remarked earlier, the smallest possible automorphism group of a Cayley graph on $A$ is $A\rtimes\langle\iota\rangle$. It is then natural to conjecture (as did Babai and Godsil~\cite[Remark~4.2]{BaGo}) that almost all Cayley graphs on $A$ have automorphism group $A\rtimes\langle\iota\rangle$. This conjecture was recently proved by Dobson and the last two authors. 

\begin{theorem}\label{th:dobson}{{ $($\cite[Theorem~1.5]{DSV}$)$}}
Let $A$ be an abelian group of order $n$. The proportion of inverse-closed subsets $S$ of $A$ such that $\Aut(\Cay(A,S)) = A\rtimes\langle\iota\rangle$ goes to $1$ as $n\to\infty$.
\end{theorem}

We now turn our attention to the other exceptional family in Conjecture~\ref{graphmainconjecture}, namely the family of generalised dicyclic groups. The situation in this case is slightly more delicate, as this exceptional family contains an even more exceptional sub-family, as we now explain.

Recall that, if $R$ is a generalised dicyclic group then every Cayley graph on $R$ admits $R\rtimes\langle\iota\rangle$ as a group of automorphisms. One might be tempted to conjecture that, as in the abelian case, almost all Cayley graphs on $R$ have $R\rtimes\langle\iota\rangle$ as their full automorphism group. It turns out that this is not the case. Indeed, if $R\cong \Q_8\times \C_2^\ell$ (where $\Q_8$ denotes the quaternion group of order $8$ and $\C_2$ the cyclic group of order $2$) then there exists a permutation group $B$ containing $R$ as a regular subgroup of index $8$ and such that every Cayley graph on $R$ admits $B$ as a group of automorphisms. (See Notation~\ref{Q-notation} for the definition of $B$ and Lemma~\ref{lemma42} for a proof of this fact.)

Our main result is that almost all Cayley graphs on generalised dicyclic groups have automorphism group as small as possible, in the sense of the previous paragraph. More precisely, we prove the following.

\begin{theorem}\label{th:main}
Let $R$ be a generalised dicyclic group of order $n$, let $B=R\rtimes\langle\iota\rangle$ if $R\ncong \Q_8\times \C_2^\ell$ and let $B$ be as in Notation~$\ref{Q-notation}$ if $R\cong \Q_8\times \C_2^\ell$. The proportion of inverse-closed subsets $S$ of $R$ such that $\Aut(\Cay(R,S)) = B$ goes to $1$ as $n\to\infty$.
\end{theorem}

Theorem~\ref{th:main} immediately follows from Theorems~\ref{Theo:main not Q} and~\ref{Theo:mainQ}. In our proof of Theorem~\ref{th:main} we do not make any effort to keep track of the error terms in our estimates. By being more careful, one may obtain the following two more explicit versions of Theorem~\ref{th:main}.

\begin{theorem}\label{th:epsilon1}
Let $R$ be a generalised dihedral group of order $n$ with $R\ncong \Q_8\times \C_2^\ell$ and let $m$ be the number of elements of order at most $2$ of $R$. Then $R$ has $2^{m/2+n/2}$ inverse-closed subsets and the number of inverse-closed subsets $S$  with $\Aut(\Cay(R,S))>R\rtimes\langle\iota\rangle$ is at most $(2^{m/2+n/2})\cdot \varepsilon$ where 
\begin{eqnarray*}
\varepsilon&=& 2^{-n/48+2(\log_2(n))^2+4}.
\end{eqnarray*}
\end{theorem}
\begin{theorem}\label{th:epsilon2}
Assume Notation~$\ref{Q-notation}$. Then $R\cong \Q_8\times \C_2^\ell$, $|R|=n$, $R$ has $2^{5n/8}$ inverse-closed subsets and the number of inverse-closed subsets $S$  with $\Aut(\Cay(R,S))>B$ is at most $(2^{5n/8})\cdot \varepsilon$ where 
\begin{eqnarray*}
\varepsilon&=&2^{-n/512+(\log_2(n))^2+2}.
\end{eqnarray*}
\end{theorem}

We now give a brief summary of the rest of the paper. In Section~\ref{sec:Preliminaries}, we establish some basic preliminary results. The case $R \ncong\Q_8\times\C_2^\ell$ of Theorem~\ref{th:main} is dealt with in Section~\ref{sec: Not QxC2} while the case $R\cong \Q_8\times \C_2^\ell$ is proven in Section~\ref{sec:Q8 times E}. In Section~\ref{sec:Sec1proofs}, we show that the corresponding version of our results for unlabelled graphs easily follows, and give a version of our results for Cayley digraphs. Finally, we use Theorem~\ref{theoPrimitive} in Section~\ref{sec: Not QxC2} but its proof is technical and of a different flavor than the rest of the paper, so it is delayed to Section~\ref{sec:Proof of Theo}.

\section{Preliminaries} \label{sec:Preliminaries}

Throughout the paper, we denote by $\C_n$ a cyclic group of order $n$, by $\D_n$ a dihedral group of order $2n$ and by $\Q_8$ the quaternion group of order $8$. We say that a group $D$ is a {\em generalised dihedral group} on $A$ if $A$ is an abelian subgroup of $D$ of index $2$ and there exists an involution $w\in D\setminus A$ with $a^w=a^{-1}$ for every $a\in A$.

\subsection{Generalised dicyclic groups}\label{Sec:GenDicyclic}
We now establish some basic properties of generalised dicyclic groups. A reader familiar with these groups can probably skip this subsection with little loss.

\begin{notation}\label{notation1}
Let $A$ be an abelian group of even order and of exponent greater than $2$, let $y$ be an involution in $A$ and let $R=\Dic(A,y,x)$.  Let $\iota:R\to R$ be the permutation of $R$ fixing $A$ pointwise and mapping every element of $R\setminus A$ to its inverse.
\end{notation}

We first mention a few  basic properties of $R$ which will be used repeatedly and without comment. (The proofs follow immediately from the definitions.)

\begin{lemma}\label{basic-R}
Assume Notation~$\ref{notation1}$. Then the following hold.
\begin{enumerate}
\item $\iota$ is an automorphism of $R$.
\item $\langle y\rangle$ is a characteristic subgroup of $R$. \label{yeye}
\item Every element in $R\setminus A$ has order $4$ and squares to $y$.
\item Every subgroup of $R$ is either abelian or generalised dicyclic.
\item The centre of $R$ consists of the elements of $A$ of order at most $2$.
\item $R\cong \Q_8 \times \C_2^\ell$ if and only if $A\cong \C_4\times \C_2^\ell$ and $y$ is the unique non-identity square in $A$. 
\end{enumerate}
\end{lemma}

We also prove some slightly less trivial results.

\begin{lemma}\label{BasicDicyclic}
Assume Notation~$\ref{notation1}$ and $R\ncong \Q_8 \times \C_2^\ell$.
\begin{enumerate}
\item Let $b\in A$ and let $X=\{a\in A\mid a^2\in\{b,by\}\}$. Then $|X|\leq 2|A|/3$. \label{yiyi}
\item Let $U<A$ and let $X=\{a\in A\mid a\notin U, a^2\neq y\}$. Then $|X|\geq |A|/4$.\label{yaya}
\end{enumerate}
\end{lemma}
\begin{proof}
Let $A_2=\{a\in A\mid a^2=1\}$. As $A$ is abelian, if $a_1,a_2\in A$ and $a_1^2=a_2^2$ then $(a_1a_2^{-1})^2=1$.  It follows that $A_2$ is a subgroup of $A$ and, since $A$ has exponent greater than $2$, in fact $A_2<A$. Moreover, it also follows that $\{a\in A\mid a^2=b\}$ and $\{a\in A\mid a^2=by\}$ are either empty or have cardinality $|A_2|$. In particular,~(\ref{yiyi}) follows when $|A:A_2|\geq 3$. We thus assume that $|A:A_2|\leq 2$ hence $|A:A_2|=2$ and $A\cong \C_4\times \C_2^\ell$.  Since $R\ncong \Q_8 \times \C_2^\ell$, $y$ is not a square in $A$. It follows that at most one of $b$ and $by$ is a square in $A$ and thus $|X|\leq |A|/2$. This concludes the proof of~(\ref{yiyi}).

If $y$ is not a square in $A$, then $X=A\setminus U$ and~(\ref{yaya}) follows immediately. We may thus assume that $y=z^2$ for some $z\in A$. The set of elements of $A$ which square to $y$ is exactly $zA_2$. We must thus show that $|U\cup zA_2|\leq 3|A|/4$. If $|A:A_2|\geq 4$ then $|U\cup zA_2|\leq |U|+|zA_2|\leq |A|/2+|A|/4=3|A|/4$. If $|A:A_2|=3$ then $A \cong  \C_3\times \C_2^\ell$, contradicting the fact that $A$ contains the element $z$ of order $4$. We may thus assume that $|A:A_2|=2$. It follows that $A\cong \C_4\times \C_2^\ell$ and, since $y$ is a square in $A$, we get $R\cong \Q_8\times \C_2^\ell$, which is a contradiction.
\end{proof}

\subsection{Primitive groups}
In this subsection, we recall some basic facts about primitive permutation groups. For terminology regarding the types of primitive groups, we follow~\cite{ONAN}.  Three types of primitive groups will be particularly important in this paper.

Let $G$ be a primitive permutation group. The group $G$ is of {\em affine type} if it contains a regular elementary abelian $p$-group $T$. In this case, $G_1$ acts faithfully and irreducibly on $T$, which is the unique minimal normal subgroup of $G$. The group $G$ is of {\em almost simple type} if $T \le G \le \Aut(T)$ for some non-abelian simple group $T$. Finally, $G$ is of \emph{product action type} if $G$ is a subgroup of the wreath product $H \wr \Sym(l)$ endowed with its natural action on $\Delta^l$ with $l\geq 2$ and $H$ an almost simple primitive group on $\Delta$. Furthermore, if $T$ is the socle of $H$ then $G$ has a unique minimal normal subgroup $N$ and $N=T_1\times\cdots\times T_l$ where $T_i\cong T$ for every $i\in\{1,\ldots,l\}$. Finally, $\norm G {T_i}$ projects surjectively onto $H$ for every $i\in\{1,\ldots,l\}$.

\begin{lemma}\label{lemma:Fr}
Let $G$ be a primitive permutation group with an abelian point-stabiliser. Then $G$ is of affine type.
\end{lemma}
\begin{proof}
See for example~\cite[Lemma~$2.1$]{DSV}.
\end{proof}

\begin{lemma}\label{lemma:cycliccenter}
Let $G$ be a primitive permutation group of affine type with point-stabiliser $G_1$ and socle $T$. Then $\Z {G_1}$ is cyclic of order coprime to $|T|$.
\end{lemma}
\begin{proof}
This follows immediately from Schur's lemma. For a complete proof, see for example~\cite[Theorem~$1$]{LiebeckS}.
\end{proof}

\begin{lemma}\label{lemma3}
A primitive permutation group of almost simple type cannot have a point-stabiliser of exponent dividing $4$.
\end{lemma}
\begin{proof}
We argue by contradiction and suppose that $G$ is a primitive permutation group of almost simple type with point-stabiliser $G_1$ a $2$-group of exponent at most $4$.  Since $G_1$ is a non-identity maximal core-free subgroup of $G$, it is self-normalising in $G$ and hence is a Sylow $2$-subgroup of $G$.

Let $T$ be the socle of $G$. Inspecting the lists in~\cite{LiZ2011} yields that $T\cong \PSL(2,q)$ for some  $q$ and that $T\cap G_1$ is a non-abelian dihedral group of order $q+1$ or $q-1$. Since $G_1$ has exponent at most $4$ and $T\cap G_1$ is non-abelian, we have $T\cap G_1\cong \D_4$. In particular $q=7$ or $q=9$. 

If $T\cong\PSL(2,7)$ then either $G\cong\PSL(2,7)$ or $G\cong\PGL(2,7)$. In the latter case, a Sylow $2$-subgroup of $G$ has exponent $8$, while in the former case, a Sylow $2$-subgroup of $G$ is not maximal in $G$. We thus obtain a contradiction in both cases.

If $T\cong\PSL(2,9)$ then $T\leq G\leq\PGammaL(2,9)$. In particular, $G$ is isomorphic to one of the following: $\PSL(2,9)$, $\PGL(2,9)$, $M_{10}$, $\Sym(6)$ or $\PGammaL(2,9)$. It is straightforward to check that in none of these cases is a Sylow $2$-subgroup of $G$ both maximal and of exponent at most $4$. This contradiction concludes the proof.
\end{proof}

\begin{corollary}\label{cor3}
A primitive permutation group with a point-stabiliser of exponent dividing $4$ is of affine type.
\end{corollary}
\begin{proof}
Assume that $G$ is not of affine type and let $G_1$ be a point-stabiliser of $G$. Again, $G_1$ is a Sylow $2$-subgroup of $G$ and hence $G$ has odd degree. In particular, by~\cite[Theorem]{Liebeck}, $G$ is of almost simple or product action type. By Lemma~\ref{lemma3}, we may assume that $G$ is of product action type. 

Let $N$ be the socle of $G$. Then $N\cong T^\ell$ where $T$ is a non-abelian simple group and $\ell \geq 2$, and $N\unlhd G\leq H\wr \Sym(\ell)$, with $T\unlhd H\leq \Aut(T)$. From the structure of primitive groups of product action type, $H$ is a primitive group of almost simple type with point-stabiliser  a $2$-group isomorphic to a quotient of  a subgroup of $G_1$, and hence of exponent dividing $4$. This contradicts Lemma~\ref{lemma3}.
\end{proof}

\subsection{Counting lemmas}

We now prove a few basic counting lemmas that will be used repeatedly.

\begin{lemma}\label{counting subgroups}
Let $G$ be a group of order $n$. The number of automorphisms of $G$ and the number of subgroups of $G$ are both at most $2^{o(n)}$.
\end{lemma}
\begin{proof}
Clearly, $G$ admits a generating set of size at most $\log_2(n)$ and hence $|\Aut(G)|\leq n^{\log_2(n)}= 2^{o(n)}$. Similarly, any subgroup of $G$ is also at most $\log_2(n)$-generated and thus $G$ has at most $n^{\log_2(n)}= 2^{o(n)}$ subgroups. 
\end{proof}

\begin{lemma}\label{win across blocks}
Let $R$ be a group of order $n$, let $m$ be the number of elements of order at most $2$ in $R$ and let $M$ be a subgroup of $R$. Then there are at most $2^{m/2+n/2-|M|/2+o(n)}$ inverse-closed subsets $S$ of $R$ such that $\Aut(\Cay(R,S))$ contains a subgroup $G$ with the following properties:
\begin{itemize}
\item $G$ contains the right regular representation of $R$,
\item $G$ normalises $M$,
\item $G_1$ centralises $M$, and
\item $G_1$ is not contained in the kernel of the action of $G$ on $M$-orbits.
\end{itemize}
\end{lemma}

\begin{proof}
Let $\Omega$ be the set of right cosets of $G_1$ in $G$. During this proof, we will be using the action of $G_1$ by conjugation on the elements of $M$; to avoid confusion with the action of $G$ on the vertex-set $R$ of $\Cay(R,S)$  we will consider $G$ as a subgroup of $\Sym(\Omega)$. Since $R$ acts regularly on $\Omega$ there is a natural bijection $\varphi$ from $\Omega$ to $R$, where $\varphi(G_1g)$ is the unique element $r\in R$ such that $G_1g=G_1r$. 

Since $G_1$ is not contained in the kernel of the action of $G$ on $M$-orbits, there exists $g \in G_1$ such that $(G_1qM)g=G_1rM$ for some $q, r \in R$ with $qM \neq rM$. In particular, $G_1qg=G_1r\bar{m}$ for some $\bar{m}\in M$. The number of choices for each of $qM$ and $rM$ is at most $n/|M|$. We now assume that $qM$ and $rM$ are fixed. The number of inverse-closed subsets of $R\setminus rM$ is at most $2^{m/2+(n-|M|)/2}$ and hence this is an upper bound for the number of choices for $(R\setminus rM)\cap S$. 

Since $g$ centralises $M$, we have $$(G_1qx)g=G_1qxg=G_1qgx=G_1r\bar{m}x$$ for every $x \in M$.  Since $g \in G_1$, $g$ must map $qM \cap S$ onto $rM\cap S$. It follows that  $rM\cap S=r\bar{m}q^{-1}(qM\cap S)$ and thus $rM\cap S$ is completely determined by $\bar{m}$ and by $qM \cap S$. The number of choices for $\bar{m}$ is at most $|M|$ and thus the number of choices for $S$ is at most \[(n/|M|)^2\cdot |M|\cdot  2^{m/2+n/2-|M|/2}\leq  2^{m/2+n/2-|M|/2+o(n)}.\qedhere\]
\end{proof}

\section{Cayley graphs on $R$ with $R\ncong \Q_8 \times \C_2^\ell$}\label{sec: Not QxC2}
We first introduce some notation.

\begin{notation}\label{notation2}
Let $A$ be an abelian group of even order and of exponent greater than $2$, let $y$ be an element of order $2$ in $A$ and let $R=\Dic(A,y,x)$. Assume that $R\ncong \Q_8 \times \C_2^\ell$. Let $\iota:R\to R$ be the automorphism of $R$ fixing $A$ pointwise and mapping every element of $R\setminus A$ to its inverse. Let $B=R\rtimes\langle\iota\rangle$, let $C=A\times \langle \iota \rangle$  and let $D=A\rtimes\langle \iota x\rangle$. Let $n=|R|$ and let $m$ be the number of elements of order at most $2$ in $R$.
\end{notation}

Note that $\iota$ fixes every inverse-closed subset of $R$ setwise and hence every Cayley graph on $R$ admits $B$ as a group of automorphisms. The main result of this section is that, in fact, almost all Cayley graphs on $R$ have $B$ as their full automorphism group. Before we state and prove this, we collect a few basic results about $B$, some of which will be used repeatedly. 

\begin{lemma}\label{basic-sec3}
Assume Notation~$\ref{notation2}$. Then 
\begin{enumerate}
\item $C$ is abelian. \label{baba}
\item Every subgroup of $A$ is normal in $B$. \label{bebe}
\item $A$, $C$, $D$ and $R$ are the only proper subgroups of $B$ containing $A$. \label{bibi}
\item $D$ is a  generalised dihedral group on $A$ and $A$ is characteristic in $D$.\label{bobo}
\item If $X\leq R$, $b\in B$ and $X^b\cap X=1$, then $X=1$. \label{bubu}
\item $R$ is characteristic in $B$.\label{RCharB}
\item Let $N$ be a subgroup of index $2$ in $B$ such that $N \notin \{C,D\}$. Then $y\in N$ and the orbit of $y$ under $\Aut(N)$ has size at most $2$.\label{yCharN}
\end{enumerate}
\end{lemma}

\begin{proof}
The proofs of (\ref{baba}--\ref{bobo}) follow immediately from the definitions.

Proof of~(\ref{bubu}): by~(\ref{bebe}), $X\cap A$ is normal in $B$ and thus $X\cap A=(X\cap A)^b\leq X^b\cap X=1$. Since $|R:A|=2$, it follows that $|X|\leq 2$. As every element of $R\setminus A$ has order $4$, we have $X\leq A$ and thus $X\cap A=X=1$.

Proof of~(\ref{RCharB}): by contradiction, suppose that $R'$ is a distinct conjugate subgroup of $R$ in $B$. Since neither $C$ nor $D$ is generalised dicyclic, it follows that $R'\notin \{C,D\}$. By~(\ref{bibi}), this implies that $A\nleq R'$. Since $|B:R'|=2$, it follows that $|X:X\cap R'|=2$ for every $X\in\{A,C,D,R\}$.

Let $d\in (D\setminus A)\cap R'$. Note that $d$ is an involution. Since every involution in $R$ is central, the same holds in $R'$ and thus $d$ is central in $R'$. Note that $C\cap R'$ is an abelian subgroup of index $2$ in $R'$ and that $C\cap R'$ and $d$ generate $R'$. It follows that $R'$ is abelian, which is a contradiction.

Proof of~(\ref{yCharN}): let $N$ be a subgroup of index $2$ in $B$ such that $N \notin \{C,D\}$. By~(\ref{bibi}), it follows that $N\cap (R\setminus A)\neq \emptyset$ and $N\cap A$ has index at most $2$ in $A$. 

Note that the elements of $R\setminus A$ square to $y$ and, in particular, $y\in N$.  Note also that all the elements of $D\setminus A$ square to the identity. It follows that any square in $B$ distinct from $1$ and $y$ has all of its square roots in $C$, and, since $C$ is abelian, all of these square roots commute with each other. 

In particular, if the square roots in $N$ of $y$ do not commute, then $y$ is the unique non-identity square in $N$ whose square roots do not commute, and hence the orbit of $y$ under $\Aut(N)$ has size $1$. We may thus assume that the square roots in $N$ of $y$ commute. 

Fix $ax \in N \cap (R\setminus A)$ and let $b \in N \cap A$. Note that $ax$ and $bax$ are both square roots in $N$ of $y$ and hence $(ax)(bax)=(bax)(ax)$.  With a computation, this yields $b^2=1$ and hence $N\cap A$ is an elementary abelian $2$-group. 

Since $N\cap A$ has index at most $2$ in $A$ and $A$ is not an elementary abelian $2$-group, it follows that $A\cong \C_4\times \C_2^i$ for some $i$ and thus $C\cong \C_4\times \C_2^{i+1}$. This implies that $C$ contains a unique non-identity square $z$. Thus $y$ and $z$ are the only (not necessarily distinct) non-identity squares of $B$, and hence the orbit of $y$ under $\Aut(N)$ has size at most $2$. This completes the proof.
\end{proof}

The following lemma will also prove useful.

\begin{lemma}\label{2blocks}
Assume Notation~$\ref{notation2}$. Then there are at most $2^{m/2+23n/48+o(n)}$ inverse-closed subsets $S$ of $R$ such that  $\Aut(\Cay(R,S))$ contains a subgroup $H$ with the following properties:
\begin{itemize}
\item $A\leq H$,
\item the $A$-orbits are $H$-invariant, and 
\item $|H:A|$ does not divide $4$.
\end{itemize}
\end{lemma}
\begin{proof}
Let $\iota'$ be the automorphism of $A$ mapping every element to its inverse and let $A'=A\rtimes \langle\iota'\rangle$. 

Suppose first that $\Aut(\Cay(A,A\cap S))> A'$. By~\cite[Proof of Theorem~1.5]{DSV}, there are at most $2^{m/2+11n/48+2(\log_2(n))^2+2}=2^{m/2+11n/48+o(n)}$ possible choices for $A\cap S$ with this property. Since $S$ is inverse-closed and no element of $R\setminus A$ is an involution, there are at most $2^{n/4}$ choices for $(R\setminus A)\cap S$.  Thus altogether there are at most $2^{m/2+23n/48+o(n)}$ possible choices for $S$ in this case.

We now consider the case when $\Aut(\Cay(A,A\cap S))=A'$. Let $H_A$ be the stabiliser in $H$ of the $A$-orbit $A$ and let $\Lambda$ be the group induced by the action of $H_A$ on $A$. Since $A\leq H_A$ and since $H$ acts as a group of automorphisms of $\Cay(R,S)$, we  have $A\leq \Lambda\leq \Aut(\Cay(A,A\cap S))=A'$. 

From the Embedding Theorem~\cite[Theorem~$1.2.6$]{meldrum}, $H \le \Lambda \wr \C_2=(\Lambda\times \Lambda)\rtimes\C_2$. (The first coordinate corresponds to the action on $A$ while the second coordinate corresponds to the action on $xA$.) Moreover, under this embedding, the group $A$ is identified with the diagonal subgroup $\{(z,z)\mid z\in A\}$ and  $H_A\leq \Lambda\times \Lambda$.  Let $K=\{(z_1,z_2)\in H_A\mid z_1=1\}$ and let $L=\{(z_1,z_2)\in H_A\mid z_1\in\{1,\iota'\}\}$.

We claim that there exists $z=(z_1,z_2)\in L$ with $z_2\notin \{1,y\}$. Assume, on the contrary, that $z_2\in \{1,y\}$ for every $(z_1,z_2)\in L$. In particular, $|K|\leq 2$. On the other hand, we have
$$|H:A|=|H:H_A||H_A:A|=|H:H_A||\Lambda:A||K|.$$ 
Since $A$ has two orbits which are $H$-invariant, we have $|H:H_A|\leq 2$. As  $|H:A|$ does not divide $4$ and $|\Lambda:A|\leq 2$, it follows that $|H:H_A|=|\Lambda:A|=|K|=2$. We may thus assume that $K=\langle(1,y)\rangle$, that $H$ is transitive, and that $\Lambda=A'$. 

Since $\Lambda=A'$, it follows that $|L|=2|K|=4$. As we are assuming that every element in $L$ has second coordinate in $\{1,y\}$, we get $L=\langle(1,y),(\iota',1)\rangle$. Since $H$ is transitive, there exists $h\in H$ interchanging the two $A$-orbits. As the first coordinate of $(\iota',1)^h$ is the identity, we get $(\iota',1)^h\in K$ and hence $(\iota',1)^h=(1,y)$.  Observe that $\iota'$ fixes some but not all of the points of $A$, thus $(\iota',1)^h=(1,y)$ fixes some but not all of the points of $xA$. This is a contradiction since $y$ acts fixed point-freely on $xA$. This completes the proof of our claim.

\smallskip

There are at most $2|A'|\leq 2^{o(n)}$ choices for $z$. We now assume that $z=(z_1,z_2)$ is fixed and count the number of elements $ax \in R\setminus A$ such that $(ax)^z\in \{ax,(ax)^{-1}\}=\{ax,axy\}$. First, suppose that $z_2\in A$. In this case, we have $(ax)^z=axz_2$ and $axz_2\in \{ax, axy\}$ if and only if $z_2\in \{1,y\}$, which is a contradiction. Next, suppose that $z_2=b\iota'$ for some $b\in A$. We have $(ax)^z=(axb)^{\iota'}=(ab^{-1}x)^{\iota'}=a^{-1}bx$ and $a^{-1}bx\in\{ax, axy\}$ if and only if $a^2\in\{b,by\}$. By Lemma~\ref{BasicDicyclic}(\ref{yiyi}), the number of such $a$ is at most $2|A|/3$. In particular, there are at least $|A|/3=n/6$ elements $a\in A$ such that $(ax)^z\notin \{ax,(ax)^{-1}\}$.

Since $z_1\in\{1,\iota'\}$, it follows that $z$ fixes the vertex of $\Cay(R,S)$ corresponding to the identity and thus $S$ is $\langle z\rangle$-invariant. As $z\in H_A$, in fact $(R\setminus A)\cap S$ is $\langle z\rangle$-invariant and hence clearly $\langle z,\iota\rangle$-invariant as well. Since $R\setminus A$ does not contain any involutions, any element of $R\setminus A$ is in an orbit of length at least $2$ under $\langle z,\iota\rangle$. Furthermore, observe that if $(ax)^z\notin\{ax,(ax)^{-1}\}$, then $ax$ is in an orbit of length at least $4$ under $\langle z,\iota\rangle$.  It follows that the number of choices for $(R\setminus A)\cap S$  is at most $2^{\frac{n/6}{4}+\frac{n/2-n/6}{2}}=2^{5n/24}$.  As the number of choices for $A\cap S$ is at most $2^{m/2+n/4}$, there are at most $2^{m/2+11n/24+o(n)}$ choices for $S$ in this case.

Adding the results we obtained in the two cases, we find that the number of choices for $S$ is at most  $2^{m/2+23n/48+o(n)}$.
\end{proof}

Finally, we will need the following result. As the proof is long, technical, and different in flavour from the rest of the paper, we will present it separately in Section~\ref{sec:Proof of Theo}.

\begin{theorem}\label{theoPrimitive}
Assume Notation~$\ref{notation2}$. Let $X=B/N$ be a quotient of $B$ and let $G$ be a primitive permutation group with point-stabiliser $X$. Then $G$ has a unique minimal normal subgroup. Moreover, either $G$ is of affine type or $y\in N$.
\end{theorem}

We are now ready to prove the main theorem of this section.

\begin{theorem}\label{Theo:main not Q}
Assume Notation~$\ref{notation2}$.  The number of inverse-closed subsets $S$ of $R$ such that $\Aut(\Cay(R,S))>B$ is at most $2^{m/2+23n/48+o(n)}$. In particular, the proportion of inverse-closed subsets $S$ of $R$ such that $\Aut(\Cay(R,S))=B$ goes to $1$ as $n\to\infty$.
\end{theorem}

\begin{proof}
Note that the number of inverse-closed subsets of $R$ is $2^{m/2+n/2}$, hence the second part of the theorem follows from the first. Let $S$ be an inverse-closed subset of $R$ such that $\Aut(\Cay(R,S))>B$ and let $G$ be a subgroup of $\Aut(\Cay(R,S))$ containing $B$ as a maximal subgroup. 

\smallskip
\noindent\textbf{Case 1.} $y$ is central in $G$.
\smallskip

Let $T=\{s \in S\mid sy \not\in S\}$ and let $U =\langle T\rangle$.  Since $S$ is $G_1$-invariant and $y$ is central in $G$, it follows that $T$ is $G_1$-invariant. Hence the $U$-orbits are $G_1$-invariant.
Note that $\iota\in G_1$ and, for every $z\in R\setminus A$, we have $z^\iota=zy$. Since $S$ is $G_1$-invariant, this shows that $T \cap (R\setminus A)=\emptyset$, and hence $T \subseteq A$ and $U \le A$.  By Lemma~\ref{basic-sec3}(\ref{bebe}), $R$ normalises every subgroup of $A$. In particular $R$ normalises $U$ and thus the $U$-orbits are invariant under $G=RG_1$.

Suppose first that $U=A$. Since $|B:A|=4$ and $G>B$ we have $|G:A|>4$.  Therefore Lemma~\ref{2blocks} (applied with $H=G$) implies that there are at most $2^{m/2+23n/48+o(n)}$ choices for $S$ in this case.

We now assume that $U<A$. By Lemma~\ref{counting subgroups}, there are at most $2^{o(n)}$ choices for $U$. Assume that $U$ is fixed and let $g$ be the permutation of $R$ that fixes every element of $U$ and every element of $R\setminus A$ but interchanges $a$ with $ay$ for every $a \in A\setminus U$. A few calculations reveal that $g$ is an automorphism of $\Cay(R,S)$ fixing the identity. Let $X=\{r \in R\mid \{r,r^{-1}\}^g\neq\{r,r^{-1}\} \}$ and note that, if $r\in A\setminus U$ and $r^2\neq y$ then $r\in X$. Since $U<A$, it follows from Lemma~\ref{BasicDicyclic}(\ref{yaya}) that $|X|\ge |A|/4$.  
Let $m_1$ be the number of involutions in $X$. Note that $X$ is itself inverse-closed.  The number of inverse-closed subsets of $R\setminus X$ is exactly $$2^{(m-m_1)+((n-|X|)-(m-m_1))/2}=2^{(m-m_1)/2+(n-|X|)/2}$$ and hence this is an upper bound for the number of choices for $(R\setminus X)\cap S$.  

Involutions of $R$ lying in $X$ are in $\langle g \rangle$-orbits of length at least $2$; while if $x$ is a non-involution in $X$, then $\{x,x^{-1}\}^g\neq \{x,x^{-1}\}$ and hence the smallest $\langle g \rangle$-invariant inverse-closed subset of $X$ containing $x$ has size at least $4$. Since $S$ is inverse-closed and $\langle g \rangle$-invariant, the number of choices for $X\cap S$ is at most $2^{m_1/2+(|X|-m_1)/4}$. Hence the total number of choices for $S$ is at most $2^{m/2+n/2-|X|/4-m_1/4} \le 2^{m/2+n/2-|X|/4}$. Since $|X|\ge |A|/4$, it follows that there are at most $2^{m/2+15n/32+o(n)}$ choices for  $S$ in this case.

\smallskip
\noindent\textbf{Case 2.} $y$ is not central in $G$.
\smallskip

Let $N$ be the core of $B$ in $G$, that is $N=\bigcap_{g \in G} B^g$. We will use the ``bar convention'' and, for all $X \le G$,  denote the group $XN/N$ by $\overline{X}$. It follows from Lemma~\ref{basic-R}(\ref{yeye}) and Lemma~\ref{basic-sec3}(\ref{RCharB}) that $\langle y\rangle$ is characteristic in $B$. Since $y$ is not central in $G$, $B$ is not normal in $G$ and thus $N<B$. As $B$ is maximal in $G$, we have that $\overline{G}$ is a primitive permutation group on the cosets of $\overline{B}$, with point-stabiliser $\overline{B}$. 

Suppose that $N\not\le C$. Fix $cx\in N\cap (B\setminus C)$ where $c \in C$. For every $a\in A$ we have $(cx)^a=cx^a=cxa^2 \in N$ thus $a^2\in N$ and hence $A^2 \le N$. Moreover $(cx)(cx)^x=cxx^{-1}cxx=c^2y \in N$.  Since $C^2=A^2 \le N$, this shows that $y \in N$ and hence $\langle A^2, y \rangle \le N$. This implies that $\overline{B}$ is an elementary abelian $2$-group.  By Lemmas~\ref{lemma:Fr} and~\ref{lemma:cycliccenter}, it follows that $\overline{B}$ is cyclic of order $2$. Since the point-stabiliser of a primitive group is self-normalising, $\overline{B}$ is a Sylow $2$-subgroup of $\overline{G}$ and $|G:B|$ is odd. 

Suppose that $N\neq D$. Since $|B:N|=2$, it follows by Lemma~\ref{basic-sec3}(\ref{yCharN}) that $y\in N$ and the orbit of $y$ under $\Aut(N)$ has size at most $2$. Since $|G:B|$ is odd and $y$ is central in $B$, this implies that $y$ is central in $G$, which is a contradiction.

We may thus assume that $N=D$. Since $A$ is characteristic in $D$, it follows that $A$ is normal in $G$. By Lemma~\ref{2blocks} (applied with $H=G$), there are at most $2^{m/2+23n/48+o(n)}$ choices for  $S$ in this case.

From now on, we assume that $N\leq C$. By Theorem~\ref{theoPrimitive}, $\overline{G}$ has a unique minimal normal subgroup $T/N$. 

Suppose that $N<C$ and that $y\in N$. Since $C$ is abelian, we have $N<C\leq\cent G N\unlhd G$ and thus $1<\overline{\cent G N}\unlhd \overline{G}$. As $T/N$ is the unique minimal normal subgroup of $\overline{G}$, it follows that $T\leq \cent G N$. In particular, $y$ is centralised by $T$. As $y$ is central in $B$, $y$ is central in $BT=G$, which is a contradiction. 

We may thus assume that either $N=C$ or $y\notin N$. If $N=C$ then $|\overline{B}|=2$ and thus $\overline{G}$ is dihedral and $\overline{T}$ is odd. If $y\notin N$ then, by Theorem~\ref{theoPrimitive}, $\overline{G}$ is of affine type. Since $y\notin N$, it follows that the centre of $\overline{B}$ has even order and thus, by Lemma~\ref{lemma:cycliccenter}, $|\overline{T}|$ is odd.

In both cases, we have obtained that $\overline{G}$ is of affine type and $|\overline{T}|$ is odd. It follows that $|T:N|=|CT:C|$ is odd. Since $C$ has two orbits, namely $A$ and $R\setminus A$, and these have the same size, $CT$ is intransitive with the same orbits as $C$. As $|G:CT|=|B:C|=2$, $CT$ is normal in $G$ and it follows from  Lemma~\ref{2blocks} (applied with $H=CT$) that there are at most $2^{m/2+23n/48+o(n)}$ choices for  $S$ in this case.

Adding the results we obtained in the four cases, we find that the number of choices for $S$ is at most  $2^{m/2+23n/48+o(n)}$.
\end{proof}

\section{Cayley graphs on $\Q_8\times \C_2^\ell$}\label{sec:Q8 times E}

Before stating the main result of this section, we need to introduce some notation.

\begin{notation}\label{Q-notation}
Let $E$ be an elementary abelian $2$-group, let $\Q_8=\langle i,j\mid i^4=1, i^2=j^2,  i^j=i^{-1}\rangle$ and let $R=\Q_8\times E$. We label the elements of $\Q_8$ with $\{1,-1,i,-i,j,-j,k,-k\}$ in the usual way. Let $M = \langle-1 \rangle\times E$, let $n=|R|$ and let $m$ be the number of elements of order at most $2$ in $R$. We define the following permutations of $R$: for $\ell \in \{i,j,k\}$, $\alpha_\ell$  is the involution that swaps $\ell e$ and$-\ell e$ for every $e\in E$ and fixes every other element of $R$. Let  $B=\langle R, \alpha_i, \alpha_j,\alpha_k\rangle$, viewed as a permutation group on $R$ (with $R$ acting regularly on itself by right multiplication). 
\end{notation}

The importance of $B$ in this context can be seen with the following observation.

\begin{lemma}\label{lemma42}
Assume Notation~$\ref{Q-notation}$. Every Cayley graph on $R$ admits $B$ as a group of automorphisms.
\end{lemma}
\begin{proof}
Let $S$ be an inverse-closed subset of $R$, let  $x,y\in R$, let $s=xy^{-1}$ and write $x=qe$ and $y=rf$ with $q,r\in\Q_8$ and $e,f\in E$. Note that
$$ 
x^{\alpha_i}(y^{\alpha_i})^{-1}=
\begin{cases}
-s&\textrm{if } |\{q,r\}\cap\{i,-i\}|=1,\\
s&\textrm{otherwise.}
\end{cases}
$$
Moreover, if $|\{q,r\}\cap\{i,-i\}|=1$ then $-s=s^{-1}$ and thus $x^{\alpha_i}(y^{\alpha_i})^{-1}\in\{s,s^{-1}\}$ in all cases. This implies that $\alpha_i$ is an automorphism of $\Cay(R,S)$. By an analogous argument, the same is true for $\alpha_j$ and $\alpha_k$ and the result follows.
\end{proof}

The main result of this section is that almost all Cayley graphs on $R$ have $B$ as their full automorphism group. Before we state and prove this, we collect a few basic results about $B$ which will be  used repeatedly. The proofs are left to the reader.

\begin{lemma}\label{Q-basic}
Assume Notation~$\ref{Q-notation}$. Then 
\begin{enumerate}
\item $M=\Z B$, 
\item $|B:R|=8$, 
\item $B$ has exponent $4$, 
\item $m=|M|=n/4$, and
\item $R$ has exactly $2^{5n/8}$ inverse-closed subsets. \label{tata}
\end{enumerate}
\end{lemma}

\begin{theorem}\label{Theo:mainQ}
Assume Notation~$\ref{Q-notation}$.  The number of inverse-closed subsets $S$ of $R$ such that $\Aut(\Cay(R,S))>B$ is at most $2^{5n/8-n/512+o(n)}$. In particular, the proportion of inverse-closed subsets $S$ of $R$ such that $\Aut(\Cay(R,S))=B$ goes to $1$ as $n\to\infty$.
\end{theorem}
\begin{proof}
By Lemma~\ref{Q-basic}(\ref{tata}), the number of inverse-closed subsets of $R$ is $2^{5n/8}$, hence the second part of the theorem follows from the first.

Let $S$ be an inverse-closed subset of $R$ such that $\Aut(\Cay(R,S))>B$, let $G$ be a subgroup of $\Aut(\Cay(R,S))$ containing $B$ as a maximal subgroup and let $g\in G_1\setminus B_1$.  As $B$ is maximal in $G$, we have $G=\langle B,g\rangle$.

\smallskip
\noindent\textbf{Case 1.} $M$ is normal in $G$.
\smallskip

Suppose that $M$ is not centralised  by $G_1$. In particular, there is some element of $G_1$ the action of which by conjugation on $M$ induces a non-trivial automorphism $\phi$ of $M$.  Since $\phi$ fixes at most half of the elements of $M$, the number of subsets of $M$ which are $\phi$-invariant is at most $2^{|M|/2}2^{|M|/4}=2^{3|M|/4}=2^{3n/16}$. By Lemma~\ref{counting subgroups}, there are at most $2^{o(n)}$ choices for $\phi$ and thus at most $2^{3n/16+o(n)}$ choices for $M \cap S$. Since all of the involutions of $R$ are in $M$, the number of inverse-closed subsets of $R\setminus M$ is $2^{(n-|M|)/2}=2^{3n/8}$. This is an upper bound for the number of choices for $(R\setminus M)\cap S$. The number of choices for $S$ is thus at most $2^{9n/16+o(n)}$ in this case.

We now assume that $M$ is centralised by $G_1$. Suppose that $G_1$ is not contained in the kernel of the action of $G$ on $M$-orbits. By Lemma~\ref{win across blocks} there are at most $2^{m/2+n/2-|M|/2+o(n)}=2^{n/2+o(n)}$ choices for $S$ in this case.

We now assume that $G_1$ fixes every $M$-orbit setwise. In what follows, for $r\in R$ and $h\in G$, we write $r^h$ to denote the image of the vertex $r$ under the permutation $h$, and $h^{-1}rh$ to denote the conjugate of the permutation $r\in G$ by the element $h$. Let $i^g=iq_i$, $j^g=jq_j$ and $k^g=kq_k$, where $q_i,q_j,q_k \in M$.  Since $g$ centralises $M$, we have $(\ell m)^g=\ell^g m=\ell q_\ell m$ for every $m \in M$ and every $\ell\in \{i,j,k\}$. It follows that $g$ is determined by $q_i$, $q_j$ and $q_k$ and thus there are at most $|M|^3\leq 2^{o(n)}$ choices for $g$.

Suppose that $q_i,q_j,q_k\in \{-1,1\}$. Note that $\alpha_i,\alpha_j\in B_1$. By replacing $g$ by an element of $\{g,g\alpha_i,g\alpha_j, g\alpha_i\alpha_j\}$, we may assume that $q_i=q_j=1$.  Since $g\neq 1$, we have $q_k=-1$ and hence $g=\alpha_k$, contradicting the fact that $g\notin B_1$. 

We may thus assume that $q_{\bar{\ell}} \notin \{-1,1\}$ for some $\bar{\ell}\in \{i,j,k\}$. It follows that, for every $m\in M$, we have $(\bar{\ell}m)^g\notin \{\bar{\ell}m,(\bar{\ell}m)^{-1}\}$. Since $\bar{\ell}M\cap S$ is inverse-closed and preserved by $g$, there are at most $2^{|\bar{\ell}M|/4}=2^{n/16}$ choices for $\bar{\ell}M\cap S$. As the number of choices for $M\cap S$ is at most $2^{n/4}$ and the number of choices for $\ell M\cap S$ is at most $2^{n/8}$ for each $\ell\in\{i,j,k\}\setminus\{\bar{\ell}\}$, there are at most $2^{9n/16+o(n)}$ choices for $S$ in this case.

\smallskip
\noindent\textbf{Case 2.} $M$ is not normal in $G$.
\smallskip

Since $M=\Z B$, $B$ is not normal in $G$ either. Since $B$ is maximal but not normal in $G$, it must be self-normalising in $G$ and thus $\langle B,B^g\rangle = G$. Moreover, since $B$ is a $2$-group, it follows that $B$ must be a Sylow $2$-subgroup of $G$ and hence $|G:B|$ is odd.  In particular, by the orbit-stabiliser theorem, $G_1$ is not a $2$-group.

Let  $H=M\cap M^g$.  Since $M= \Z B$, $H$ is central in $B$. By the same reasoning, $H$ is central in $B^g$ and thus in  $\langle B,B^g\rangle=G$.

Let $N$ be the core of $B$ in $G$. Note that $H\leq N< B$ and, since $B$ is maximal in $G$, we can view $G/N$ as a primitive permutation group with point-stabiliser $B/N$. As $B$ has exponent $4$, $B/N$ has exponent dividing $4$ and hence, by Corollary~\ref{cor3}, $G/N$ is of affine type. Moreover, since $MN/N$ is contained in the centre of $B/N$, it follows from Lemma~\ref{lemma:cycliccenter} that $MN/N$ is cyclic. Since $M$ is an elementary abelian $2$-group, so is $MN/N$ and hence $|MN:N|\leq 2$. 

As $|B:M| = 32$, we have $|MN:M|\leq 32$. It follows that $|N:N\cap M|\leq 32$  and hence $|N\cap M^g:N\cap M\cap M^g|=|N\cap M^g:H|\leq 32$. Applying $g^{-1}$ to $N\cap M^g$ and $H$, we obtain $|N\cap M:H|\leq 32$. As $|MN:N|\leq 2$, we have $|M:N\cap M|\leq 2$ and hence $|M:H|=|M:N\cap M||N\cap M:H|\leq 2\cdot 32=64$.  Finally, $|R:M|=4$ and hence $|R: H| \le 256$.

Suppose that $G_1$ is contained in the kernel of the action of $G$ on $H$-orbits. It follows that $HG_1$ is normal in $G$. Since $H$ is central in $G$ and $H\cap G_1=1$, we have $HG_1=H\times G_1$. As $HG_1$ is normal in $G$, we have $\langle z^2\mid z\in HG_1\rangle=\langle z^2\mid z\in G_1\rangle$ is normal in $G$ and hence $\langle z^2\mid z\in G_1\rangle=1$ because $G_1$ is core-free in $G$. Thus $G_1$ is an elementary abelian $2$-group, which is a contradiction. 

We may thus assume that $G_1$ is not contained in the kernel of the action of $G$ on $H$-orbits. It then follows from Lemma~\ref{win across blocks} that there are at most $2^{m/2+n/2-|H|/2+o(n)}=2^{5n/8-n/512+o(n)}$ choices for $S$ in this case.

Adding the results we obtained in the four cases, we find that the number of choices for $S$ is at most  $2^{5n/8-n/512+o(n)}$.
\end{proof}

Combined together, Theorems~\ref{Theo:main not Q} and~\ref{Theo:mainQ} yield Theorem~\ref{th:main}.

\section{Related results: Unlabelled graphs, and digraphs}\label{sec:Sec1proofs}

An  \emph{unlabelled} graph is simply an isomorphism class of (labelled) graphs. We often identify a representative with its class. An easy consequence of Theorems~\ref{Theo:main not Q} and~\ref{Theo:mainQ} is the following unlabelled version of Theorem~\ref{th:main}.

\begin{theorem}\label{th:MaIN}
Let $R$ be a generalised dicyclic group of order $n$, $B=R\rtimes\langle\iota\rangle$ if $R\ncong \Q_8\times \C_2^\ell$ and let $B$ be as in Notation~$\ref{Q-notation}$ if $R\cong \Q_8\times \C_2^\ell$. The proportion of unlabelled Cayley graphs $\Gamma$ over $R$ such that $\Aut(\Gamma) = B$ goes to $1$ as $n\to\infty$.
\end{theorem}
\begin{proof}
Let $\Gamma_1=\Cay(R,S_1)$ with $\Aut(\Gamma_1)=B$. We show that the number of inverse-closed subsets $S_2$ of $R$ such that $\Gamma_1\cong\Cay(R,S_2)$ is at most $2^{o(n)}$. The result then follows from Theorems~\ref{Theo:main not Q} and~\ref{Theo:mainQ}.

Let $\Gamma_2=\Cay(R,S_2)$ and let $\varphi$ be a graph isomorphism from $\Gamma_1$ to $\Gamma_2$. Note that $\varphi$ induces a group isomorphism, $\phi$ say, from $\Aut(\Gamma_1)=B$ to $\Aut(\Gamma_2)=B$ and hence $\phi \in \Aut(B)$.  

Note that there exists a characteristic subgroup $X$ of $B$ such that $X\leq R$ and $|B:X|\leq 32$. (If $R\ncong \Q_8 \times \C_2^\ell$ then take $X=R$ and use Lemma~\ref{basic-sec3}(\ref{RCharB}), if $R\cong \Q_8\times \C_2^\ell$ then take $X=M$ as in Notation~\ref{Q-notation}.) It follows that $X\leq R^\phi$. Since  $|B:X|\leq 32$, it follows that $B$ has at most $O(1)$ subgroups containing $X$ and thus there are at most $O(1)$ choices for the subgroup $R^\phi$. As $|\Aut(R)|\leq 2^{o(n)}$, there are at most $2^{o(n)}$ choices for an isomorphism from $R$ to a given $R^\phi$ and thus at most $2^{o(n)}$ choices for $\phi$ (and hence for $\varphi$ and $S_2$).
\end{proof}

One can define \emph{Cayley digraphs} in the obvious way: if $S$ is a  (not necessarily inverse-closed) subset of a group $R$ then $\Cay(R,S)$ is the digraph with vertex-set $R$ and with $(g,h)$ being an arc if and only if $gh^{-1}\in S$. Our proof of Theorem~\ref{th:main} with a few minor adjustments yields the corresponding directed version.

\begin{theorem}\label{th:MAIN}
Let $R$ be a generalised dicyclic group of order $n$. The proportion of subsets $S$ of $R$ such that $\Aut(\Cay(R,S)) = R$ goes to $1$ as $n\to\infty$.
\end{theorem}

\section{Proof of Theorem~\ref{theoPrimitive}}\label{sec:Proof of Theo}

This section is dedicated solely to the proof of Theorem~\ref{theoPrimitive}. We first need a few preliminary results.

\begin{lemma}\label{prim2}
A primitive permutation group with generalised dicyclic point-stabilisers is of affine type.
\end{lemma}
\begin{proof}
Let $R=\Dic(A,y,x)$ be a point-stabiliser of the primitive group $G$. Suppose first that $R \cap R^g=1$ for every $g \in G\setminus R$. Then, $G$ is a Frobenius group with complement $R$. Since $G$ is primitive, the Frobenius kernel is an elementary abelian group and $G$ is of affine type. We may thus assume that there exists $\bar{g}\in G\setminus R$ with $R\cap R^{\bar{g}}\neq 1$.

Since $R$ is a maximal subgroup of $G$, $R$ is self-normalising in $G$. It follows that $R^{\bar{g}} \neq R$ and hence $\langle R,R^{\bar{g}}\rangle=G$. Since every subgroup of $A$ is normal in $R$, $A \cap A^{\bar{g}}\trianglelefteq R$. For the same reason, $A \cap A^{\bar{g}}\trianglelefteq R^g$ and thus $A \cap A^{\bar{g}}\trianglelefteq G$. Since $A \cap A^{\bar{g}}$ is contained in the point-stabiliser $R$, it follows that $A \cap A^{\bar{g}}=1$.

In particular,  we have either $R\cap R^{\bar{g}}\nleq A$ or $R\cap R^{\bar{g}}\nleq A^{\bar{g}}$. By symmetry, we may assume that the former holds. As $|R:A|=2$, we get $R=A(R\cap R^{\bar{g}})$ hence $ax\in R\cap R^{\bar{g}}$ for some $a\in A$. The square of every element of $R^{\bar{g}}$ lies in $A^{\bar{g}}$ thus $(ax)^2=y\in A^{\bar{g}}$. Since $y\in A$, this contradicts the fact that $A\cap A^{\bar{g}}=1$.
\end{proof}

\begin{lemma}\label{hell}Assume Notation~$\ref{notation2}$. A primitive permutation group with point-stabiliser $B$ is of affine type.
\end{lemma} 
\begin{proof}
Let $G$ be a primitive group with point-stabiliser $B$ and, towards a contradiction, suppose that $G$ is not of affine type. 

Suppose that $x\in B\cap B^g$ for some $g\in G\setminus B$. Since the square of an element of order $4$ in $B$ in necessarily central in $B$, $x^2$ is central in $B$. A similar reasoning implies that $x^2$ is central in $B^g$ and thus in $\langle B,B^g\rangle=G$. Since $B$ is core-free in $G$ and $1\neq x^2=y\in B$, we reach a contradiction. 

It follows that for every $g\in G\setminus B$, $x\notin B\cap B^g$. In particular, $x$ fixes only one point. Since $x$ has order $4$, this implies that $G$ has odd degree. It then follows from~\cite[Theorem]{Liebeck} that $G$ is either of almost simple or of product action type. In particular, $G$ has a unique minimal normal subgroup $N$.

\smallskip
\noindent\textbf{Case 1.} $G$ is of almost simple type.
\smallskip

In this case, the structure of $N$ and $B\cap N$ are described in~\cite[Theorem, Part~(b)]{Liebeck}. In order to use this classification more effectively, we first make some observations about $B\cap N$.

It follows  from~\cite[Theorem~$1$]{LiebeckS} that $\Z B$ is cyclic. Since all elements of order at most $2$ in $A$ are central in $B$, it follows that $y$ is the unique involution in $A$ and the Sylow $2$-subgroup of $A$ is cyclic. Let $a$ be a generator of the Sylow $2$-subgroup of $A$,  let $2^\ell$ be the order of $a$ and let $B_2=\langle a,x,\iota\rangle$. Clearly, $B_2$ is a Sylow $2$-subgroup of $B$, and hence of $G$. Thus $B_2\cap N$ is a Sylow $2$-subgroup of $N$. 

Suppose that $B_2\cap N\leq R=\langle A,x\rangle$, that is, $B_2\cap N\leq \langle a,x\rangle$. Observe that $\langle a,x\rangle$ is a generalised quaternion group. Thus $B_2\cap N$ is either cyclic or a generalised quaternion group. By~\cite{BrauerSuzuki}, a non-abelian simple group cannot have a cyclic or generalised quaternion group Sylow $2$-subgroup, which is a contradiction.

Suppose that $B_2\cap N\leq C=A\times \langle\iota\rangle$, that is, $B_2\cap N\leq \langle a\rangle\times\langle \iota\rangle$.  Then $N$ has an abelian Sylow $2$-subgroup and hence, by the remarkable theorem of Walter~\cite{Walter}, $N$ is isomorphic either to $\PSL(2,2^f)$ for some $f\geq 3$, to $\PSL(2,q)$ for some $q\equiv 3,5\pmod 8$, to the Janko  group $J_1$, or to a Ree group $\mathrm{Ree}(3^{2m+1})$ for some $m\geq 1$. Since $B_2\cap N$ is $2$-generated, a quick inspection reveals that $N\cong \PSL(2,q)$  for some $q\equiv 3,5\pmod 8$. In particular, $|B_2\cap N|=4$ and hence $B_2\cap N=\langle y\rangle\times\langle\iota\rangle$. As $q\equiv 3,5\pmod 8$, $q$ is not a square and hence $\Out(N)\cong\Out(\PSL(2,q))$ is cyclic of odd order. Since $G/N$ is isomorphic to a subgroup of $\Out(N)$, so is $B_2N/N\cong B_2/(B_2\cap N)$ and hence $B_2\leq N$. This is a contradiction since $B_2$ has order at least $8$ while $B_2\cap N$ has order $4$.

We conclude that $B_2\cap N\leq\langle a,\iota,x\rangle$ but $B_2\cap N\nleq \langle a, \iota\rangle$ and $B_2\cap N\nleq \langle a,x\rangle$. Combining this information with the description of the primitive almost simple groups of odd degree in~\cite[Theorem, Part~(b)]{Liebeck} and the classification of the maximal solvable subgroups of almost simple groups in~\cite{LiZ2011} yields that $N\cong\PSL(2,q)$, and either $B\cap N\cong \D_{(q+1)/2}$ and $q\equiv 3\pmod 4$, or $B\cap N\cong \D_{(q-1)/2}$ and $q\equiv 1\pmod 4$. 

If $q\in \{5,7,9\}$ then the conclusion follows by computation (no group $G$ with $N\leq G\leq \PGammaL(2,q)$ has a Sylow $2$-subgroup isomorphic to a Sylow $2$-subgroup of $B$). We thus assume that $q\geq 11$ and write $\varepsilon=1$ if  $q\equiv 1\pmod 4$ and $\varepsilon=-1$ if $q\equiv 3\pmod 4$.  Since $A\times \langle\iota\rangle$ is an abelian subgroup of index $2$ in $B$, $(A\times \langle\iota\rangle)\cap N$ is an abelian subgroup of index at most $2$ in $B\cap N=\D_{(q-\varepsilon)/2}$. Let  $A_0=(A\times \langle\iota\rangle)\cap N$. As $q-\varepsilon\geq 10$, $\D_{(q-\varepsilon)/2}$ has a unique abelian subgroup of index at most $2$ and hence $A_0 \cong \C_{(q-\varepsilon)/2}$ and, in particular, $A_0$ is a maximal torus of $N$. After two computations, one for the case $\varepsilon=1$ and one of the case $\varepsilon=-1$, we see that $\cent{\PGammaL(2,q)}{A_0}\leq \PGL(2,q)$, and hence $A\times \langle\iota\rangle\leq \PGL(2,q)$. Since $\cent{\PGammaL(2,q)}{A_0}=\cent{\PGL(2,q)}{A_0}$ is a maximal torus of $\PGL(2,q)$ of order $q-\varepsilon$, we obtain that $A\times\langle\iota\rangle$ is cyclic. This implies that $A$ has odd order, which is a contradiction.

\smallskip
\noindent\textbf{Case 2.} $G$ is of product action type.
\smallskip

In particular, $N\unlhd G\leq H\wr \Sym(\ell)$ with $\ell\geq 2$, $H$ an almost simple group with socle $T$ and with $N\cong T^\ell$. Let $N=T_1\times \cdots \times T_\ell$ with $T_i\cong T$ for every $i\in \{1,\ldots,\ell\}$. 

 For every $i\in \{1,\ldots,\ell\}$, let $B_i=B\cap T_i$. From the structure of primitive permutation groups of product action type~\cite{ONAN}, we have $B\cap N=B_1\times \cdots \times B_\ell$ with $|B_1|=\cdots =|B_\ell|>1$. As $N$ is transitive, we have $G=NB$. It follows that $B$ acts transitively by conjugation on the set $\{T_1,\ldots,T_\ell\}$ and thus on $\{B_1,\ldots,B_\ell\}$ and, since $R\unlhd B$, also on $\{(B_1\cap R),\ldots,(B_\ell\cap R)\}$. In particular, the groups $B_1\cap R,\ldots,B_\ell\cap R$ are pairwise conjugate in $B$ and pairwise intersect trivially. Since $\ell\geq 2$, it follows from Lemma~\ref{basic-sec3}(\ref{bubu}) that $B_1\cap R=\cdots =B_\ell\cap R=1$.  As $|B:R|= 2$, we have $|B_i|= 2$ for every $i\in \{1,\ldots,\ell\}$ and hence $B\cap N=B_1\times \cdots \times B_\ell$ is an  elementary abelian $2$-group. 

Since $B=\norm G{{B\cap N}}$, it follows that $B\cap N=\norm N {B\cap N}$ and $B_i=\norm {T_i}{B_i}$ for every $i\in \{1,\ldots,\ell\}$. Since $B_i$ has order $2$ and is self-normalising in $T_i$, it must be a Sylow $2$-subgroup of $T_i$. This is a contradiction since a non-abelian simple group cannot have a Sylow $2$-subgroup of order $2$ (see~\cite[7.2.1]{KurStell} for example).
\end{proof}

We now prove Theorem~\ref{theoPrimitive} which we restate, for convenience.

\smallskip

\noindent\textbf{Theorem~\ref{theoPrimitive}.} \emph{Assume Notation~$\ref{notation2}$. Let $X=B/N$ be a quotient of $B$ and let $G$ be a primitive permutation group with point-stabiliser $X$. Then $G$ has a unique minimal normal subgroup. Moreover, either $G$ is of affine type or $y\in N$.}

\vspace{-.1cm}

\begin{proof}
Observe that $X$ is solvable. The finite primitive permutation groups with solvable point-stabilisers are classified in~\cite{LiZ2011}. From~\cite[Theorem~$1.1$]{LiZ2011} we see that $G$ is either of affine,  almost simple or product action type. In particular, $G$ has a unique minimal normal subgroup. Suppose that $y\notin N$. We show that $G$ is of affine type.

By Lemma~\ref{lemma:Fr} we may assume that $X$ is non-abelian. If $N\nleq C$ then $B=NC$ and $X=B/N\cong C/(C\cap N)$ is abelian, which is a contradiction. We may thus assume that $N\leq C$. 

If $N\nleq R$ then $B=NR$ and $X=B/N\cong R/(R\cap N)$. Since $X$ is non-abelian and $y\notin N$, it follows that $X$ is a generalised dicyclic group and hence $G$ is of affine type by Lemma~\ref{prim2}. We may thus assume that $N\leq R$ and hence $N\leq C\cap R=A$.

As $y\notin N$, we obtain that $R/N$ is isomorphic to the generalised dicyclic group $\Dic(A/N,yN,xN)$, $X=B/N$ is isomorphic to a group as in Notation~\ref{notation2} and hence $G$ is of affine type by Lemma~\ref{hell}.
\end{proof}

\end{document}